\documentclass[letter, conference]{IEEEconf}
\IEEEoverridecommandlockouts
\usepackage{cite}
\usepackage{amsmath,amssymb,amsfonts}

\usepackage{parskip} 
\begingroup
    \makeatletter
    \@for\theoremstyle:=definition,remark,plain\do{%
        \expandafter\g@addto@macro\csname th@\theoremstyle\endcsname{%
            \addtolength\thm@preskip\parskip
            }%
        }
\endgroup

\usepackage[absolute]{textpos}
\usepackage{balance}
\usepackage{graphicx}
\usepackage{overpic}
\usepackage{rotating}
\usepackage{algorithm,algorithmic}
\usepackage{mathtools}
\newcommand{\algorithmicbreak}{\textbf{break}}
\newcommand{\BREAK}{\STATE \algorithmicbreak}
\usepackage{mymath}

\makeatletter
\renewcommand*\env@matrix[1][\arraystretch]{%
  \edef\arraystretch{#1}%
  \hskip -\arraycolsep
  \let\@ifnextchar\new@ifnextchar
  \array{*\c@MaxMatrixCols c}}
\makeatother

\newtheorem{lemma}{Lemma}
\newtheorem{definition}{Definition}
\newtheorem{theorem}{Theorem}
\newtheorem{corollary}{Corollary}

\def\uselargeplots{1}

\usepackage{graphicx}
\usepackage{textcomp}
\usepackage{xcolor}
\def\BibTeX{{\rm B\kern-.05em{\sc i\kern-.025em b}\kern-.08em
    T\kern-.1667em\lower.7ex\hbox{E}\kern-.125emX}}

\begin{document}

\title{\LARGE \bf ADMM for Block Circulant Model Predictive Control}

\author{Idris Kempf${^*}$, Paul J.\ Goulart and Stephen Duncan
\thanks{$^*$Corresponding author: {\tt\footnotesize{idris.kempf@eng.ox.ac.uk}}. All authors are with the Department of Engineering Science, University of Oxford, Oxford, UK. This research is supported by the Engineering and Physical Sciences Research Council (EPSRC) with a Diamond CASE studentship. }
}

\maketitle

\begin{abstract}
This paper deals with model predictive control problems for large scale dynamical systems with cyclic symmetry. Based on the properties of block circulant matrices, we introduce a complex-valued coordinate transformation that block diagonalizes and truncates the original finite-horizon optimal control problem. Using this coordinate transformation, we develop a modified alternating direction method of multipliers (ADMM) algorithm for general constrained quadratic programs with block circulant blocks. We test our modified algorithm in two different simulated examples and show that our coordinate transformation significantly increases the computation speed.
\end{abstract}


\section{Introduction}
The advantages of model predictive control (MPC) for constraint handling and feedforward disturbance modelling are widely recognised. However, its applicability is limited by the requirement to solve optimization problems in real-time to compute the control law. This constraint has inhibited the application of MPC to large-scale and high speed applications. While some approaches for accelerating the computing speed have focused on implementing optimization routines on specialised high-performance hardware \cite{FPGAMPC}, other approaches have exploited the particular symmetric structure encountered in some classes of large-scale problems \cite{SYMMETRICMPC}. In this paper, we address systems with cyclic symmetry resulting in a block circulant structure. These systems can be interpreted as the symmetric interconnection of many subsystems, where each subsystem interacts in an identical way with its neighbors \cite{ANDREADISTRIBUTED}. Circulant systems can be found in a variety of applications, including vehicle formation control \cite{PHDTRAFFIC,CIRCBALLOON}, cross-directional control \cite{PAPERMACHINES}, particle accelerator control \cite{PAROBUST} and in the approximation of partial differential equations \cite{PDE}. The mathematical properties of these systems have already been exploited in controller design \cite{HCIRC,DISTRCIRC}, stability analysis \cite{STABILITYCIRC} and subspace identification \cite{CIRCSYSID}. In this paper, the properties that a constrained quadratic program (CQP) inherits from a block circulant MPC problem are investigated. The main results of the paper show how exploiting the properties of the resulting CQP can reduce the computational cost when it is solved using the alternating direction of multipliers method \cite{ADMMBOYD}.

 This paper is structured as follows. In Section \ref{sec:problemstatement}, the linear model predictive control (MPC) problem and the alternating direction of multipliers method (ADMM) -- an algorithm which is particularly suitable for solving the latter optimization problem -- are introduced. Since this paper is concerned with the analysis of systems with block circulant symmetry, we introduce the notion of block circulant matrices in Section \ref{sec:circulantdecomposition}. Furthermore, the block circulant MPC problem is formally defined and necessary conditions for its decomposition are stated. In Section  \ref{sec:blockcirculantadmm}, we define a CQP with block circulant blocks and show how MPC problems with block circulant data can be written in this form. The block circulant decomposition is then applied to the CQP and a modified ADMM algorithm is then formulated for this problem. In Section \ref{sec:simulations}, we compare the performance of the original and modified ADMM algorithms. For the sake of comparison, both algorithms have been implemented in Matlab and tested on two illustrative examples.

 \textit{Notation and Definitions} The set of real and complex numbers is denoted by $\R$ and $\C$, respectively. Let $\otimes$ denote the Kronecker product and $\oplus$ denote the direct sum (i.e.\ the block diagonal concatenation) of two matrices. Let $\I_n$ represent the identity matrix in $\R^{n\times n}$.  For a scalar, vector or matrix $a$, let $\bar{a}$ denote its complex conjugate; Let $Re(a)$ and $Im(a)$ denote its real and imaginary part, respectively; Let $a^H$ denote its Hermitian transpose. Let $\diag\lbrace a_1,\dots,a_n\rbrace$ denote a diagonal matrix with diagonal elements $a_1,\dots,a_n$.
\section{Problem Statement}\label{sec:problemstatement}
\subsection{Model Predictive Control}\label{sec:mpc}
Given a discrete-time linear dynamical system and an initial condition $x(t)$ at time $t$, a standard MPC scheme computes a control law by predicting the future evolution of the system and minimizing a quadratic objective function over some planning horizon $T$. This can achieved via repeated solution of the following quadratic program (QP):
\begin{subequations}
\begin{align}
\min &\sum_{k=0}^{T-1} x_k^\Tr Q x_k+u_k^\Tr R u_k + x_N^\Tr P x_N\label{eq:mpcA}\\
s.t.\,\,\, &x_{k+1} = Ax_k+Bu_k, \quad x_0=x(t)\label{eq:mpcB}\\
     &y_k = Cx_k+Du_k\label{eq:mpcC}\\
     &\ubar{y} \leq y_k \leq \bar{y}\label{eq:mpcD}
\end{align}\label{eq:mpc}\end{subequations}
for $k=0,\dots,T-1$ and outputting the first input vector $u_0$ of the optimal control law. The constraints on the states $x_k\in\R^{n_x}$ and the inputs $u_k\in\R^{n_u}$ are lumped into the variable $y_k\in\R^{n_y}$. The prediction horizon is $T$ and the dynamics of the system are described by \eqref{eq:mpcB}. The stability of the state is guaranteed if $P=P^\Tr\succ 0$ is obtained from the discrete-time algebraic Riccati equation (DARE),
\begin{align}\label{eq:dare}
A^\Tr P A - A^\Tr P B \left(B^\Tr P B +R\right)^{-1}B^\Tr P A + Q =P.
\end{align}
The problem \eqref{eq:mpc} has a unique solution if $R\succ 0$, $Q\succeq 0$ and the pairs $(A,B)$ and $(A,Q^{\frac{1}{2}})$ are controllable and observable, respectively \cite[Chapter 12]{MPCBOOK}. Throughout the paper, we will assume that $P$ is obtained from \eqref{eq:dare} and that \eqref{eq:mpc} admits a unique solution.

 By eliminating the state variables $(x_1, \dots, x_N)$ and defining  $z \eqdef (u_0,\dots,u_{T-1})^\Tr$ and $v\eqdef(y_0,\dots,y_{T-1})^\Tr$, \eqref{eq:mpc} can be reformulated as
\begin{subequations}\label{eq:mpcqp}
\begin{align}
\min \,\,\,&\frac{1}{2}z^\Tr J z +q^\Tr z\label{eq:mpcqpA}\\
s.t.\,\,\, &Kz-v=0\label{eq:mpcqpB}\\
     &\ubar{v} \leq v \leq \bar{v},\label{eq:mpcqpC}
\end{align}
\end{subequations}
where we do not make the dependency of $(\bar v, \ubar{v}, q)$ on $x_0$ explicit for simplicity of notation.
The matrices $(J,K)$ and vectors $(\ubar{v},\bar{v},q)$ in \eqref{eq:mpcqp} are defined as
\begin{subequations}
\begin{align}
J &\eqdef G^\Tr \left( (\I_T\otimes Q)\oplus P \right)G+(\I_T\otimes R)\\
K &\eqdef \left[ \I_T\otimes C \,\vert\, 0 \right] G+(\I_T\otimes D)\\
\ubar{v} &\eqdef (\B{1}_T\otimes\ubar{y}) - \left[\I_T\otimes C \,\vert\, 0 \right]Hx_0\\
\bar{v} &\eqdef (\B{1}_T\otimes\bar{y}) - \left[\I_T\otimes C \,\vert\, 0 \right]Hx_0\\
q &\eqdef G^\Tr H x_0\label{eq:mpcqpmatricesE}
\end{align}\label{eq:mpcqpmatrices}\end{subequations}
where $\B{1}_T$ is a vector of ones of length $T$ and $G$ and $H$ arise from elimination of the equality constraints in \eqref{eq:mpcB}, i.e.\ from setting $X=(x_0,\dots,x_T)^\Tr$ and writing \eqref{eq:mpcB} as $X=Gz+Hx_0$,
\begin{align}\label{eq:dynamics}
G =
\begin{bmatrix}
0  & \dots \\
B  &   &   \\
AB & B &   &  \\
\vdots & & \ddots\\
A^{T-1}B & A^{T-2}B & \dots & B
\end{bmatrix},\quad
H=
\begin{bmatrix}
\I_{n_x} \\ A \\ \vdots \\ A^T
\end{bmatrix}.
\end{align}
Note that $J\succ 0$ because $R\succ 0$ by assumption.

\subsection{ADMM Algorithm}\label{sec:admm}
We consider application of the alternating direction of multipliers method (ADMM) to the solution of \eqref{eq:mpcqp}, and will follow the specific ADMM formulation presented in \cite{ADMMBOYD} throughout.   The method is summarized in Algorithm \ref{alg:admm}.
The augmented Lagrangian for \eqref{eq:mpcqp} can be written as
\begin{align}\label{eq:lang}
  \begin{split}
L(z,v,\gamma) = &\frac{1}{2}z^\Tr J z + q^\Tr z+\frac{\rho}{2}\lVert Kz-v\rVert_2^2\\&+\gamma^\Tr(Kz-v) + \mathcal{I}_{[\ubar{v},\bar{v}]}(v),
\end{split}
\end{align}
where $\mathcal{I}_{[\ubar{v},\bar{v}]}$ is the indicator function for the set $\set{v}{\ubar{v}\le v \le \bar{v}}$ and  the penalty parameter $\rho>0$ and the dual variables $\gamma$ are associated with the constraint \eqref{eq:mpcqpB}. ADMM solves \eqref{eq:mpcqp} by repeatedly minimizing \eqref{eq:lang} w.r.t.\ $z$ and $v$ and updating the dual variables $\gamma$ using an approximate gradient ascent method. Even though the assumptions in section \ref{sec:mpc} guarantee the convergence of Algorithm~\ref{alg:admm}, it is common practice to limit it to a maximum number of iterations $i_{max}$.
 \begin{algorithm}
 \caption{ADMM for MPC}\label{alg:admm}
 \begin{algorithmic}[1]
 \renewcommand{\algorithmicrequire}{\textbf{Input:}}
 \renewcommand{\algorithmicensure}{\textbf{Output:}}
 \REQUIRE State $x(t)$
 \ENSURE  Input $u(t)$ 
  \STATE Set $x_0=x(t)$ and $v^0,\gamma^0=0$; compute $\ubar{v},\bar{v}$ and $q$
  \FOR {$i = 1$ to $i_{max}$}
  	\STATE Update $z^i$ using (SP1)
  	\STATE Update $v^i$ using (SP2)
  	\STATE Update $\gamma^i$ using (SP3)
  	\IF {$\twonorm{v^i-v^{i-1}}^2 < \epsilon$ and $\twonorm{\gamma^i-\gamma^{i-1}}^2 < \epsilon$}
  		\BREAK
  	\ENDIF
  \ENDFOR
 \RETURN $u(t)=(z_1,\dots,z_{n_u})^\Tr$ 
 \end{algorithmic} 
 \end{algorithm}

After initialization\footnote{Note that we assume that ADMM is \emph{cold-started} in Algorithm~\ref{alg:admm} at each time step for simplicity, but in practice one would warm start the variables $(v^0,\gamma^0)$ from a previous solution.}, Algorithm \ref{alg:admm} first minimizes \eqref{eq:lang} w.r.t. to $z$, which, after completing the square, is equivalent to
\begin{align}\label{eq:minz}
z^i =\argmin_z \frac{1}{2}z^\Tr J z + q^\Tr z +\frac{\rho}{2}\lVert Kz-v^{i-1} +\rho^{-1}\gamma^{i-1}\rVert_2^2,
\end{align}
with iteration index $i=1,\dots,i_{max}$. Since \eqref{eq:minz} is an unconstrained QP, its derivative can be set to zero and the resulting linear system can then be solved from
\begin{align}\label{eq:sp1}
\left( J+\rho K^\Tr K \right) z^i = K^\Tr (\rho v^{i-1}-\gamma^{i-1}) -q.\tag{SP1}
\end{align}
The linear system \eqref{eq:sp1} always admits a solution because $J+\rho K^\Tr K\succ 0$ under the assumptions from section \ref{sec:mpc}.

 With $z^i$ obtained, Algorithm \ref{alg:admm} then minimizes \eqref{eq:lang} w.r.t. $v$ by solving
\begin{align}\label{eq:sp2}
v^i = &\argmin_{\ubar{v} \leq v \leq \bar{v}} \twonorm{Kz^i -v +\rho^{-1}\gamma^{i-1}}^2.
\end{align}
The solution to \eqref{eq:sp2} can be written as
\begin{align}\label{eq:sp2_sat}
v^i = \sat_{[\ubar{v},\bar{v}]}\left\lbrace Kz^i+\rho^{-1}\gamma^{i-1}\right\rbrace,\tag{SP2}
\end{align}
where the saturation function limits its argument to $\ubar{v}$ and $\bar{v}$.

Finally, algorithm \ref{alg:admm} updates the dual variable $\gamma$ according to
\begin{align}\label{eq:sp3}
\gamma^i = \gamma^{i-1} +\rho(Kz^i-v^i). \nonumber\tag{SP3}
\end{align}
Subproblems \eqref{eq:sp1} - \eqref{eq:sp2_sat} are repeated until some convergence criterion or the maximum number of iterations is reached. Proofs and other variants of the ADMM can be found in \cite{ADMMBOYD, OSQPMAIN}.

\section{Circulant Decomposition}\label{sec:circulantdecomposition}

\subsection{Preliminaries}\label{sec:circ_matrices}
\begin{definition}[Circulant Matrices]\label{def:circ_matrices} Let $\inR{\Cn}{n}{n}$ denote the set of real invertible \emph{circulant matrices}, i.e.\ all invertible $n\times n$ matrices in the form
\begin{align}
C = \CircC,
\end{align}
where each row is a cyclic shift of the previous row and each $c_i \in \R$.
\end{definition}
 Circulant matrices have a number of very useful basic properties \cite[Chapter 3]{CIRCBOOK}, including
\begin{subequations}
\begin{align}
C\in \Cn, \alpha \in \R &\longleftrightarrow \alpha C \in \Cn,\label{eq:prop1}\\
C\in \Cn &\longleftrightarrow C^T \in \Cn,\label{eq:prop2}\\
C\in \Cn &\longleftrightarrow C^{-1} \in \Cn,\label{eq:prop3}\\
A,C\in \Cn &\longleftrightarrow AC \in \Cn,\label{eq:prop4}\\
A,C\in \Cn &\longleftrightarrow A+C \in \Cn\label{eq:prop5}.
\end{align}\label{eq:propcirc}
\end{subequations}
\begin{definition}[Fourier Matrix]\label{def:fourier_matrix} Let $\inC{F_n}{n}{n}$ denote the \emph{Fourier matrix}, defined as
\begin{align}
F_n &= \frac{1}{\sqrt{n}}\begin{bmatrix} w_0 & w_1 & \dots & w_{n-1}\end{bmatrix}\label{eq:fourier_matrix}
\end{align}
where the vectors $w_j = \begin{pmatrix}1&\rho_j & \rho_j^2 & \dots & \rho_j^{n-1} \end{pmatrix}^\Tr$ are mutually orthogonal and $\rho_j = e^{i\frac{2\pi}{n}j}$ are complex roots of unity.
\end{definition}
 Because the vectors $w_i$ composing the matrix $F_n$ in \eqref{eq:fourier_matrix} are orthogonal and $\Vert w_i \Vert_2 = \sqrt{n}$, the matrix $F_n$ is orthogonal and unitary, i.e. $F_nF_n^H=F_n^HF_n=\I_n$.  The matrix \eqref{eq:fourier_matrix} is called the Fourier matrix since the Fourier coefficients of the discrete Fourier transformation of a vector $x \in \R^n$ can be obtained from the product $F_n x_n$ (or more efficiently using a fast Fourier transformation \cite{ROSE}).

 Perhaps the most remarkable property of circulant matrices is that every circulant matrix of order $n$ is diagonalized by the same Fourier matrix $F_n$.
\begin{theorem}[Diagonalization of $C\in\Cn$ {\cite[Chapter 3]{CIRCBOOK}}]\label{thm:diag_cn} For $\inR{C}{n}{n}$, it holds that $F_n^HCF_n$ is diagonal iff $C\in\Cn$. The diagonal elements $\lambda_0,\dots,\lambda_{n-1}$ of $F_n^HCF_n$ are
\begin{align}
\lambda_j = c_o + c_1\rho_j+\dots+c_{n-1}\rho_j^{n-1},\,\,j=0,\dots,n-1.
\end{align}
\end{theorem}
 From the structure of $F_n$, the following corollary can be established on the structure of the eigenvalues $\lambda_0,\dots,\lambda_{n-1}$: ~

\begin{corollary}[Pattern of Complex Conjugates \cite{CIRCSYSID}]\label{thm:pattern_cn} For $C\in\Cn$ the diagonal elements $\lambda_0,\dots,\lambda_{n-1}$ of $F_n^HCF_n$ have the following pattern of complex conjugates. If $n$ is odd, then $\lambda_0$ is real while $(\lambda_1,\dots,\lambda_{\frac{n-1}{2}}) = (\bar{\lambda}_{n-1},\dots,\bar{\lambda}_{\frac{n+1}{2}})$. If $n$ is even, then $\lambda_0$ and $\lambda_\frac{n}{2}$ are real while $(\lambda_1,\dots,\lambda_{\frac{n}{2}-1})=(\bar{\lambda}_{n-1},\dots,\bar{\lambda}_{\frac{n}{2}+1})$. If $C=C^\Tr$, then $\lambda_0,\dots,\lambda_{n-1}$ are real.
\end{corollary}
 Note that the same pattern of complex conjugates applies to the elements of $F_n x$ for any $x \in \R^n$.
\begin{definition}[Block Circulant Matrices] Let $\BCnpmC \subseteq \R^{np\times nm}$ denote the set of real block circulant matrices in the form
\begin{align}
B = \CircB
\end{align}
where the scalars $c_i$ from Definition \ref{def:circ_matrices} have been replaced by blocks $\inR{b_i}{p}{m}$.
\end{definition}
 Properties analogous to \eqref{eq:propcirc} can be established for block circulant matrices \cite[Chapter 5.6]{CIRCBOOK}, i.e.\
\begin{subequations}
\begin{align}
B\in \BCnpmC, \alpha \in \R &\longleftrightarrow \alpha B \in \BCnpmC,\label{eq:prop1_bc}\\
B\in \BCnpmC &\longleftrightarrow B^T \in \mathcal{BC}(n,m,p),\label{eq:prop2_bc}\\
B\in \mathcal{BC}(n,m,m) &\longleftrightarrow B^{-1} \in \mathcal{BC}(n,m,m),\label{eq:prop3_bc}\\
A,B\in \BCnpmC &\longleftrightarrow A+B \in \BCnpmC,\label{eq:prop4_bc}\\
A\in \BCnpmC,&\,\,\,B \in \mathcal{BC}(n,m,r)\nonumber \\ &\longleftrightarrow AB \in \mathcal{BC}(n,p,r).\label{eq:prop5_bc}
\end{align}\label{eq:propbcirc}\end{subequations}
For \eqref{eq:prop3_bc} it must hold that $\det B>0$. In \cite{CIRCSYSID}, Theorem \ref{thm:diag_cn} was extended to the block circulant case and is described by the following corollary:
\begin{corollary}[Block Diagonalization of $B\in\BCnpmC$]\label{thm:diag_bc} For $\inR{B}{np}{nm}$, it holds that $(F_n\otimes\I_p)^HB(F_n\otimes\I_m)$ is block diagonal iff $B\in\BCnpmC$. The blocks $\inC{\nu_j}{p}{m}$ of the block diagonalized matrix are obtained as
\begin{align}
\nu_j = b_o + b_1\rho_j+\dots+b_{n-1}\rho_j^{n-1},\,\,j=0,\dots,n-1.\label{eq:nu_j}
\end{align}
The pattern of complex conjugates described in Corollary \ref{thm:pattern_cn} also holds for the blocks $\nu_j$.
\end{corollary}
 Using the shuffle permutation matrix $\Pi_n^m$ from \cite{ROSE}, we can rewrite $F_n\otimes\I_m$ as
\begin{align}\label{eq:PsiPerm}
F_n\otimes\I_m = (\Pi_n^m)^\Tr (I_m \otimes F_n) \Pi_n^m
\end{align}
and see, that given a vector $x \in \R^{nm}$, $(F_n\otimes\I_m) x$ can be computed by permuting the vector, applying a sequence of $m$ fast Fourier transformations and applying the inverse permutation. This entails that $(F_n\otimes\I_m) x$ is of complexity $\co{mn\log n}$ compared to $\co{m^2n^2}$ for general matrix-vector multiplication. Note that the pattern of complex conjugates also holds for the $n$ blocks of the vector $(F_n\otimes\I_m)x$.

\subsection{Block Circulant Decomposition}\label{sec:bcirc_decomp}
We are concerned in particular with dynamic linear systems where the matrices present in \eqref{eq:mpc} are block circulant. More formally, a block circulant MPC problem is defined as follows.
\begin{definition}[Block Circulant MPC]\label{def:blockcircmpc}
Consider \eqref{eq:mpc} with $x_k\in\R^{nn_x},u_k\in\R^{nn_u}$ and $y_k\in\R^{pnn_y}$, where $p$ is a positive integer, and partition matrices $C$ and $D$ as $C=[C_1,\dots,C_p]^\Tr$ and $D=[D_1,\dots,D_p]^\Tr$, respectively. We say that \eqref{eq:mpc} is a \emph{block circulant MPC problem} of order $n$ if the following conditions holds:
\begin{align}
A,Q &\in \mathcal{BC}(n,n_x,n_x), \qquad &&B \in \mathcal{BC}(n,n_x,n_u),\nonumber\\
R &\in \mathcal{BC}(n,n_u,n_u), \qquad &&C_i \in \mathcal{BC}(n,n_y,n_x),\nonumber\\
D_i &\in \mathcal{BC}(n,n_y,n_u).\nonumber
\end{align}
for $i=1,\dots,p$.
\end{definition}
Any model satisfying the above conditions can be interpreted as a periodic interconnection of $n$ identical subsystems with identical constraints and objective function penalties. Introducing the integer $p$ allows for multiple constraint sets, e.g.\ to restrict the state $x_k$ and the input $u_k$ separately.
\begin{theorem}[Block Diagonalization of Block Circulant MPC]\label{thm:blockdiagmpc}
The matrices
\begin{align}
\psi_x &\eqdef  F_n\otimes\I_{n_x},\label{eq:psix}\\ \psi_u &\eqdef F_n\otimes\I_{n_u},\\\psi_y & \eqdef \I_p \otimes (F_n \otimes \I_{n_y}),
\end{align}
decompose the dynamics \eqref{eq:mpcB}-\eqref{eq:mpcC} of a block circulant MPC problem into
\begin{align}\label{eq:decompABCD}
&\psixy^H\ABCDthin\psixu = \begin{bmatrix}
\hat{A} & \hat{B} \\ \hat{C} & \hat{D}
\end{bmatrix},
\end{align}
where
\begin{align*}
\hat{A}&=\diag\lbrace\hat{a}_1,\dots,\hat{a}_n\rbrace,\qquad
\hat{B}=\diag\lbrace\hat{b}_1,\dots,\hat{b}_n\rbrace,\nonumber\\
\hat{C}&=\begin{bmatrix}\diag\lbrace\hat{c}_1^1,\dots,\hat{c}_n^1\rbrace\\ \vdots\\ \diag\lbrace\hat{c}_1^p,\dots,\hat{c}_n^p\rbrace\end{bmatrix},\,\,\,
\hat{D}=\begin{bmatrix}\diag\lbrace\hat{d}_1^1,\dots,\hat{d}_n^1)\\ \vdots\\ \diag\lbrace\hat{d}_1^p,\dots,\hat{d}_n^p\rbrace\end{bmatrix}\nonumber,
\end{align*}
the objective function matrices \eqref{eq:mpcA} into
\begin{align}\label{eq:decompQR}
&\psixu^H\QRthin\psixu = \begin{bmatrix}
\hat{Q} & 0\\
0       & \hat{R}\end{bmatrix},
\end{align}
where
\begin{align*}
\hat{Q}&=\diag\lbrace\hat{q}_1,\dots,\hat{q}_n\rbrace,\qquad
\hat{R}=\diag\lbrace\hat{r}_1,\dots,\hat{r}_n\rbrace,\nonumber
\end{align*}
and the terminal cost matrix $P$ into
\begin{align}\label{eq:decompP}
\hat{P} = \psi_x^H P \psi_x = \diag\lbrace\hat{p}_1,\dots,\hat{p}_n\rbrace.
\end{align}
\end{theorem}
\begin{proof}
According to Definition \ref{def:blockcircmpc}, \eqref{eq:decompABCD} and \eqref{eq:decompQR} are a direct consequence of Corollary \ref{thm:diag_bc}. The condition \eqref{eq:decompP} is proven through Lemma \ref{thm:dare}.
\end{proof}

\begin{lemma}[Decomposition of the Terminal Cost Matrix]\label{thm:dare}
The terminal cost matrix $P$ solves \eqref{eq:dare} if and only if $\hat{P} = \psi_x^H P \psi_x$ solves \eqref{eq:dare} with $A=\hat{A}$, $B=\hat{B}$, $Q=\hat{Q}$ and $R=\hat{R}$ for the decomposed system. Moreover, $\hat{P} = \diag\lbrace\hat{p}_1,\dots,\hat{p}_n\rbrace$.
\end{lemma}
\begin{proof}
Multiplying \eqref{eq:dare} with $\psi_x^H$ from the left and $\psi_x$ from the right, inserting $\psi_x^H\psi_x=\I_{nn_x}$ and $\psi_u^H\psi_u=\I_{nn_u}$ where appropriate and substituting \eqref{eq:decompABCD}-\eqref{eq:decompQR} yields
\begin{align*}
\hat{A}^H \psi^H_x P \psi_x \hat{A} - &\hat{A}^H \psi_x^H P \psi_x \hat{B} \left(\hat{B}^H \psi_x^H P \psi_x \hat{B}+\hat{R}\right)^{-1}&\nonumber\\
&\hat{B}^H \psi_x^H P \psi_x \hat{A} + \hat{Q} =\psi_x^H P \psi_x,
\end{align*}
where the properties of block circulant matrices \eqref{eq:prop1_bc}-\eqref{eq:prop5_bc} have been used. The transformation $\hat{P}=\psi^H P \psi_x$ therefore solves \eqref{eq:dare} for the decomposed system. Starting with
\begin{align*}\label{eq:decompdare}
\hat{A}^H \hat{P} \hat{A} - \hat{A}^H \hat{P} \hat{B} \left(\hat{B}^H \hat{P} \hat{B}+\hat{R}\right)^{-1}\hat{B}^H \hat{P} \hat{A} + \hat{Q} =\hat{P}
\end{align*}
and reversing substitutions \eqref{eq:decompABCD}-\eqref{eq:decompQR} yields $P=\psi_x \hat{P} \psi_x^H$. Since all matrices in \eqref{eq:decompdare} are block diagonal with $n$ blocks of size $m$, it can be solved for each for the blocks independently. It follows that $\hat{P}=\diag\lbrace\hat{p}_1,\dots,\hat{p}_n\rbrace$.
\end{proof}

 Theorem \ref{thm:blockdiagmpc} states that the periodic interconnection of $n$ identical subsystems can be decomposed into $n$ independent systems, often referred to as modal subsystems. Corollary \ref{thm:diag_bc} is then applied to the decomposed block circulant MPC problem.
\begin{corollary}[Truncation of Block Circulant MPC]\label{thm:truncation}
After the decomposition of Theorem \ref{thm:blockdiagmpc} has been applied to a block circulant MPC problem of order $n$, it is sufficient to examine the first $n/2$ (for $n$ even) or the first $(n-1)/2$ (for $n$ odd) blocks of matrices $\hat{A},\hat{B},\hat{Q},\hat{R},\hat{P}$ and $\hat{C_i},\hat{D_i}$ for $i=1,\dots,p$.
\end{corollary}
\begin{proof}
Direct consequence of Corollary \ref{thm:diag_bc}.
\end{proof}

In case the original matrices are full, i.e. they have $n^2n_{x,y,u}^2$ nonzero (real) elements with $p=1$, Theorem \ref{thm:blockdiagmpc} reduces the number of nonzero elements by $1/n$, resulting in $nn_{x,y,u}^2$ nonzero complex elements. Moreover, by exploiting the block-diagonal structure and pattern of complex conjugates, Corollary \ref{thm:truncation} states that only half of the blocks are required for the ADMM algorithm. Similarly, when a vector is projected into the Fourier domain by setting $\hat{x} = \psi_x^H x$, only the first $nn_x/2$ elements must be examined.

Before advancing to the decomposition of \eqref{eq:mpcqp}, the truncation of Corollary \ref{thm:truncation} and its counterpart operation, the augmentation, are formally defined.

\begin{definition}[Truncation and Augmentation]\label{def:truncaug}
Given $\inC{A=\diag\lbrace a_1,\dots,a_n\rbrace}{np}{nm}$, $x\in\C^{np}$ and $n$ even (odd), let $\trunc$ be the operator that extracts the first $\frac{n}{2}$ ($\frac{n-1}{2}$) blocks of $A$ and the first $\frac{n}{2}p$ rows of $x$. Conversely, let $\aug$ be the inverse operator which accepts a truncated vector $\p{x}=(x_1,\dots,x_{\frac{n}{2}})^\Tr\in\C^{\frac{n}{2}p}$ with $x_i\in\C^p$ and returns $x=(x_1,\dots,x_{\frac{n}{2}},\bar{x}_{\frac{n}{2}-1,\dots,\bar{x}_2})^\Tr\in\C^{np}$ for even $n$. For odd $n$, operator $\aug$ accepts a vector $\p{x}=(x_1,\dots,x_{\frac{n-1}{2}})^\Tr\in\C^{\frac{n-1}{2}p}$ with $x_i\in\C^p$ and returns $x=(x_1,\dots,x_{\frac{n-1}{2}},\bar{x}_{\frac{n-1}{2},\dots,\bar{x}_2})^\Tr\in\C^{np}$. Define $\aug$ in a similar way when invoked with a truncated block diagonal matrix $\p{A}=\diag\lbrace a_1,\dots,a_\frac{n}{2}\rbrace$. In addition, define
\begin{align*}
\trunc_N & \eqdef\I_N\otimes \trunc,\\
\aug_N & \eqdef \I_N\otimes \aug,
\end{align*}
i.e. the operators $\aug$ and $\trunc$ applied $N$ times.\end{definition}
 Note that $\aug$ and $\trunc$ are linear operators. In addition, it holds that $\trunc\lbrace J z\rbrace=\trunc\lbrace J \rbrace \trunc \lbrace z \rbrace$ and $\aug\lbrace \p{J} \p{z}\rbrace=\aug\lbrace \p{J} \rbrace \aug \lbrace \p{z} \rbrace$ for block diagonal matrices $J,\p{J}$ and vectors $z,\p{z}$ of appropriate sizes.

\section{Block Circulant ADMM Algorithm}\label{sec:blockcirculantadmm}

\subsection{Constrained QP with Block Circulant Blocks}\label{sec:block_qp}
\begin{definition}[Constrained Block Circulant QP]\label{def:blockcircqp}
The following real valued constrained QP
\begin{subequations}
\begin{align}
\min \,\,\,&\frac{1}{2}z^\Tr J z +q^\Tr z\\
s.t.\,\,\, &Kz-v=0\\
     &\ubar{v} \leq v \leq \bar{v}.\label{eq:blockcircqp_ineq}
\end{align}\label{eq:blockcircqp}\end{subequations}
is called a \emph{constrained block circulant QP} of order $n$ if there exists a partitioning of vectors $z$ and $v$ into $N_z$ and $N_v$ segments of lengths $l_z^1,\dots,l_z^{N_z}$ and $l_v^1,\dots,l_v^{N_v}$, respectively, that partition matrices $J$ and $K$ as
\begin{align*}
J = \begin{bmatrix}
J_{11} & \dots & J_{1N_z}\\
\vdots &       & \vdots\\
J_{N_z1} & \dots & J_{N_zN_z}
\end{bmatrix},
K = \begin{bmatrix}
K_{11} & \dots & K_{1N_z}\\
\vdots &       & \vdots\\
K_{N_v1} & \dots & K_{N_vN_z}
\end{bmatrix},
\end{align*}
for which all blocks are block circulant matrices of order $n$,
\begin{align*}
J_{kj} \in \mathcal{BC}(n,l_z^k,l_z^j),\qquad
K_{wj} \in \mathcal{BC}(n,l_v^w,l_z^j),
\end{align*}
for $k,j=1,\dots,N_z$ and $w=1,\dots,N_v$.
\end{definition}
The augmented Lagrangian for problem \eqref{def:blockcircqp} is obtained as
\begin{align}\label{eq:lagrangian}
\begin{split}
L(z,v,\gamma,\ubar{\lambda},&\bar{\lambda}) = \frac{1}{2}z^\Tr J z +q^\Tr z +\frac{\rho}{2} \twonorm{Kz-v}^{2}+\\
&\gamma^\Tr(Kz-v)+\ubar{\lambda}^\Tr(-v+\ubar{v})+\bar{\lambda}^\Tr(v-\bar{v}),
\end{split}
\end{align}
where we have introduced dual variables $\ubar{\lambda}$ and $\bar{\lambda}$ associated with the inequality constraints \eqref{eq:blockcircqp_ineq}. An optimal solution of \eqref{def:blockcircqp} is a saddle point of \eqref{eq:lagrangian} and must satisfy the following Karush-Kuhn-Tucker (KKT) conditions \cite{BOYDCONVEX},
\begin{subequations}
\begin{align}
Kz^* - v^* = 0,\quad v^*-\bar{v} \leq 0,\quad -v^*+\ubar{v} \leq 0\label{eq:pf}\tag{PF}\\
\bar{\lambda}^* \geq 0,\quad \ubar{\lambda}^* \geq 0\label{eq:df}\tag{DF}\\
\bar{\lambda}^{*\Tr}(v^*-\bar{v}) = 0,\quad \ubar{\lambda}^{*\Tr}(\ubar{v}-v^*) = 0\label{eq:cs}\tag{CS}\\
\nabla L = \begin{bmatrix}\frac{J+J^\Tr}{2}+\rho K^\Tr K & -\rho K^\Tr\\ -\rho K & \rho\end{bmatrix}
\begin{pmatrix}z^* \\ v^*\end{pmatrix} +
\begin{bmatrix}K^\Tr\\ -\I\end{bmatrix}\gamma^*\nonumber\\
+\begin{bmatrix}0\\ \I\end{bmatrix}\bar{\lambda}^* +
\begin{bmatrix}0\\ -\I\end{bmatrix}\ubar{\lambda}^*+ \begin{pmatrix}q \\ 0\end{pmatrix}=0\label{eq:stat}\tag{ST}
\end{align}\label{eq:KKToriginal}\end{subequations}
where \eqref{eq:pf} stands for primal feasibility, \eqref{eq:df} for dual feasibility, \eqref{eq:cs} complementary slackness and \eqref{eq:stat} for stationarity.
The following Corollary connects a block circulant MPC problem from Definition \ref{def:blockcircmpc} to the constrained block circulant quadratic program (CBCQP):
\begin{corollary}\label{thm:blocpmpisblockqp}
A block circulant MPC problem leads to a CBCQP with $N_z=T$, $N_v=Tp$ and $J_{kj} \in \mathcal{BC}(n,n_u,n_u)$ and $K_{wj} \in \mathcal{BC}(n,n_y,n_u)$.
\end{corollary}
\begin{proof}
This is a direct consequence of properties \eqref{eq:prop1_bc}-\eqref{eq:prop5_bc} and Definitions \eqref{eq:mpcqpmatrices} and \eqref{eq:dynamics}.
\end{proof}
 As the matrices in Definition \ref{def:blockcircqp} are composed of block circulant matrices, the results from Sections \ref{sec:circ_matrices} and \ref{sec:bcirc_decomp} suggest that there exists a coordinate transformation $(\tilde{z},\tilde{v})=(\psi_z^Hz, \psi_v^Hv)$ that block diagonalizes each block of $K$ and~$J$. However, how the complex-valued transformation affects the minimization in \eqref{eq:blockcircqp} is not immediately obvious. The following theorem answers this question:
\begin{theorem}[Decomposition of Block Circulant QP]\label{thm:decompblockqp}
Given a CBCQP of order $n$, then the following CBCQP,
\begin{subequations}
\begin{align}
\min_{\tilde{z}\in\Sc_z,\tilde{v}\in\Sc_v} \,\,\,&\frac{1}{2}\tilde{z}^H \tilde{J} \tilde{z} +\tilde{q}^H \tilde{z}\label{eq:tildeqpA}\\
s.t.\,\,\, &\tilde{K}\tilde{z}-\tilde{v}=0\label{eq:tildeqpB}\\
     &\ubar{v} \leq \psi_v\tilde{v} \leq \bar{v}\label{eq:tildeqpC}.
\end{align}\label{eq:blockcircqp_decomp}\end{subequations}
where $\tilde{q}=\psi_z^H q,\,\,\tilde{J}=\psi_z^H J \psi_z,\,\,\tilde{K}=\psi_v^H K \psi_z,\,\,\psi_j = \diag \lbrace F_n\otimes\I_{l_j^1},\dots,F_n\otimes\I_{l_j^{N_j}} \rbrace$ for $j=\lbrace z,v\rbrace$ and the sets $\Sc_j$ restrict each of the segments of $\tilde{z}$ and $\tilde{v}$ to the pattern of complex conjugates from Corollary \ref{thm:diag_bc}, is equivalent to \eqref{eq:blockcircqp} in the sense that
\begin{subequations}
\begin{gather}
z^* = \psi_z \tilde{z}^*,\quad v^* = \psi_v \tilde{v}^*,\\
\gamma^* = \psi_v \tilde{\gamma}^*,\quad \bar{\lambda}^* = \tilde{\bar{\lambda}}^*,\quad \ubar{\lambda}^* = \tilde{\ubar{\lambda}}^*,
\end{gather}\label{eq:decomp_equivalence}\end{subequations}
where $(z,v,\gamma,\bar\lambda,\ubar\lambda)^*$ and $(\tilde z,\tilde v,\tilde \gamma,\tilde{\bar\lambda},\tilde{\ubar\lambda})^*$ are primal and dual optimizers for \eqref{eq:blockcircqp} and \eqref{eq:blockcircqp_decomp}, respectively.
\end{theorem}
\begin{proof}
It is sufficient to show that using \eqref{eq:decomp_equivalence}, the KKT conditions of \eqref{eq:blockcircqp_decomp} yield \eqref{eq:KKToriginal}. The augmented Lagrangian for \eqref{eq:blockcircqp_decomp} is obtained as
\begin{align}\label{eq:lagrangian_decomp}
\begin{split}
\tilde{L}(&\tilde{z},\tilde{v},\tilde{\gamma},\tilde{\ubar{\lambda}},\tilde{\bar{\lambda}}) = \frac{1}{2}\tilde{z}^H \tilde{J} \tilde{z} +\tilde{q}^H \tilde{z} + \frac{\rho}{2} \twonorm{\tilde{K}\tilde{z}-\tilde{v}}^{2}\\
&+\tilde{\gamma}^H(\tilde{K}\tilde{z}-\tilde{v})+\tilde{\ubar{\lambda}}^\Tr(-\psi_v\tilde{v}+\ubar{v})+\tilde{\bar{\lambda}}^\Tr(\psi_v\tilde{v}-\bar{v})\\
&+\mathcal{I}_{\mathcal{S}_z}(\tilde{z})+\mathcal{I}_{\mathcal{S}_v}(\tilde{v}).
\end{split}
\end{align}
Note that because $\tilde{z}\in\Sc_z$ and $\tilde{v}\in\Sc_v$, $\psi_v\tilde{v}$ is real-valued and $\tilde{K}\tilde{z}-\tilde{v}\in\Sc_v$. The former implies that $\tilde{\bar{\lambda}}$ and $\tilde{\ubar{\lambda}}$ are real-valued. The latter implies that $\tilde{\gamma}$ must have the same pattern of complex conjugates as $\tilde{v}$, which can be seen by formulating \eqref{eq:tildeqpB} as two inequality constraints and evaluating the complementary slackness conditions. In addition, both observations entail that $\tilde{L}$ is a real-valued function. Conditions \eqref{eq:pf}, \eqref{eq:df} and \eqref{eq:cs} are recovered by substituting $\tilde{K}=\psi_v^H K \psi_z$ and \eqref{eq:decomp_equivalence}. The partial derivatives of $\tilde{L}$ are calculated as $\frac{\partial \tilde{L}}{\partial (\cdot)} = \frac{\partial \tilde{L}}{\partial Re(\cdot)}+i\frac{\partial \tilde{L}}{\partial Im(\cdot)}$ \cite[p. 798]{ADAPTIVEFILTER}. The gradient $\nabla \tilde{L}$ evaluates to
\begin{align}\label{eq:gradienttilde}
\begin{split}
\begin{bmatrix}\frac{\tilde{J}+\tilde{J}^H}{2}+\rho \tilde{K}^H \tilde{K} & -\rho \tilde{K}^H\\ -\rho \tilde{K} & \rho\end{bmatrix}
\begin{pmatrix}\tilde{z}^* \\ \tilde{v}^*\end{pmatrix} +
\begin{bmatrix}\tilde{K}^H\\ -\I\end{bmatrix}\tilde{\gamma}^*\\
+\begin{bmatrix}0\\ \psi_v^H\end{bmatrix}\tilde{\bar{\lambda}}^* +
\begin{bmatrix}0\\ -\psi_v^H\end{bmatrix}\tilde{\ubar{\lambda}}^*+ \begin{pmatrix}\tilde{q} \\ 0\end{pmatrix}=0.
\end{split}\tag{$\tilde{\text{ST}}$}
\end{align}
Pre-multiplying \eqref{eq:gradienttilde} by $\diag\lbrace \psi_z,\psi_v\rbrace$ and inserting claim \eqref{eq:decomp_equivalence} completes the proof.
\end{proof}
 Now that CBCQP \eqref{eq:blockcircqp} has been block diagonalized, the following theorem uses the pattern of complex conjugates from Corollary \ref{thm:diag_bc} in order to truncate problem \eqref{eq:blockcircqp_decomp}.
\begin{theorem}[Truncation of Block Circulant QP]\label{thm:truncblockqp}
Given a CBCQP of order $n$ which has been decomposed according to Theorem \ref{thm:decompblockqp}, then the following CBCQP,
\begin{subequations}
\begin{align}
\min \,\,\,&\frac{1}{2}Re(\hat{z}^H \hat{J} \hat{z}) +Re(\hat{q}^H \hat{z})\label{eq:hatqpA}\\
s.t.\,\,\, &\aug_{N_v}\lbrace\hat{K}\hat{z}-\hat{v}\rbrace=0\label{eq:hatqpB}\\
     &\ubar{v} \leq \psi_v\aug_{N_v}\lbrace\hat{v}\rbrace \leq \bar{v}\label{eq:hatqpC}.
\end{align}\label{eq:blockcircqp_trunc}\end{subequations}
where $\hat{q}=\trunc_{N_z}\tilde{q},\,\,\hat{J}=\trunc_{N_z}\tilde{J}$ and $\hat{K}=\trunc_{N_v}\tilde{K}$, is equivalent to \eqref{eq:blockcircqp_decomp} in the sense that
\begin{subequations}
\begin{gather}
\hat{z}^* = \trunc_{N_z} \tilde{z}^*,\quad \hat{v}^* = \trunc_{N_v} \tilde{v}^*,\\
\hat{\gamma}^* = \tilde{\gamma}^*,\quad \hat{\bar{\lambda}}^* = \tilde{\bar{\lambda}}^*,\quad \hat{\ubar{\lambda}}^* = \tilde{\ubar{\lambda}}^*,
\end{gather}\label{eq:trunc_equivalence}\end{subequations}
where $(\hat{z},\hat{v},\hat{\gamma},\hat{\bar\lambda},\hat{\ubar\lambda})^*$ and $(\tilde z,\tilde v,\tilde \gamma,\tilde{\bar\lambda},\tilde{\ubar\lambda})^*$ are primal and dual optimizers for \eqref{eq:blockcircqp_trunc} and \eqref{eq:blockcircqp_decomp}, respectively.
\end{theorem}
\begin{proof}
As for the proof of Theorem \ref{thm:decompblockqp}, one can show that problem \eqref{eq:blockcircqp_trunc} and \eqref{eq:blockcircqp_decomp} satisfy the same KKT conditions under claim \eqref{eq:trunc_equivalence}. Note that it is necessary to augment the equality constraints \eqref{eq:hatqpB} in order to obtain a real-valued Lagrangian:
\begin{align}\label{eq:lagrangian_trunc}
\begin{split}
\hat{L}(\hat{z},\hat{v},&\hat{\gamma},\hat{\ubar{\lambda}},\hat{\bar{\lambda}}) =\frac{1}{2}Re(\hat{z}^H \hat{J} \hat{z}) +Re(\hat{q}^H \hat{z})\\
&+ \frac{\rho}{2} \twonorm{\hat{K}\hat{z}-\hat{v}}^{2}+\hat{\gamma}^H\aug_{N_v}\lbrace\hat{K}\hat{z}-\hat{v}\rbrace\\
&+\hat{\ubar{\lambda}}^\Tr(-\psi_v\aug_{N_v}\hat{v}+\ubar{v})+\hat{\bar{\lambda}}^\Tr(\psi_v\aug_{N_v}\hat{v}-\bar{v}).
\end{split}
\end{align}
According to Definition \ref{def:truncaug}, $\aug_{N_v}\lbrace\hat{v}\rbrace\in\Sc_v$ and  $\aug_{N_z}\lbrace\hat{z}\rbrace\in\Sc_z$. Conditions PF, DF and CS are recovered using $\tilde{K}=\aug_{N_v}\lbrace \hat{K}\rbrace$ and \eqref{eq:trunc_equivalence}. The gradient of the augmented Lagrangian of the truncated problem \eqref{eq:blockcircqp_trunc} is obtained as
\begin{align}\label{eq:gradienthat}
\begin{split}
\begin{bmatrix}\frac{\hat{J}+\hat{J}^H}{2}+\rho \hat{K}^H \hat{K} & -\rho \hat{K}^H\\ -\rho \hat{K} & \rho\end{bmatrix}
\begin{pmatrix}\hat{z}^* \\ \hat{v}^*\end{pmatrix} +
\begin{bmatrix}\hat{K}^H\\ -\I\end{bmatrix}\trunc_{N_v}\hat{\gamma}^*\\
+\begin{bmatrix}0\\ \trunc_{N_v}\lbrace \psi_v^H \hat{\bar{\lambda}}^* -\psi_v^H \hat{\ubar{\lambda}}^*\rbrace\end{bmatrix}+ \begin{pmatrix}\hat{q} \\ 0\end{pmatrix}=0.
\end{split}\tag{$\hat{\text{ST}}$}
\end{align}
Augmenting both rows of \eqref{eq:gradienthat} yields \eqref{eq:gradienttilde} and completes the proof.
\end{proof}

\subsection{ADMM for Block Circulant MPC}\label{sec:circ_admm}
The version of Algorithm \ref{alg:admm} for a CBCQP, or equivalently a block circulant MPC problem, is outlined in Algorithm \ref{alg:circ_admm} and the individual steps are presented in the following paragraphs.
 \begin{algorithm}
 \caption{ADMM for Block Circulant MPC}\label{alg:circ_admm}
 \begin{algorithmic}[1]
 \renewcommand{\algorithmicrequire}{\textbf{Input:}}
 \renewcommand{\algorithmicensure}{\textbf{Output:}}
 \REQUIRE State $x(t)$
 \ENSURE  Input $u(t)$ 
  \STATE Set $x_0^0=x(t)$ and $\hat{y}^0,\hat{\gamma}^0=0$; compute $\ubar{v},\bar{v}$ and $\hat{q}$
  \FOR {$i = 1$ to $i_{max}$}
  	\STATE Update $\hat{z}^i$ using $(\hat{\text{SP1}})$
  	\STATE Update $\hat{v}^i$ using $(\hat{\text{SP2}})$
  	\STATE Update $\trunc_{N_v}\hat{\gamma}^i$ using $(\hat{\text{SP3}})$
  	\IF {$\twonorm{\aug_{N_v}\lbrace\hat{v}^i-\hat{v}^{i-1}\rbrace}^2 < \epsilon$ and\\$\qquad\twonorm{\hat{\gamma}^i-\hat{\gamma}^{i-1}}^2 < \epsilon$}
  		\BREAK
  	\ENDIF
  \ENDFOR
 \RETURN $u(t)=\psi_u \aug\lbrace (\hat{z}_0,\dots,\hat{z}_{\frac{n-1}{2}n_u})^\Tr \rbrace$ 
 \end{algorithmic} 
 \end{algorithm}
Before entering the loop, Algorithm \ref{alg:circ_admm} requires the mapping of $q$ into the complex Fourier domain.

 Next, the algorithm solves \eqref{eq:sp1} projected onto the Fourier domain,
\begin{align}\label{eq:sp1hat}
\left( \hat{J}+\rho \hat{K}^H\hat{K}\right) \hat{z}^i = \hat{K}^H (\rho \hat{v}^{i-1}-\trunc_{N_v}\hat{\gamma}^{i-1}) -\hat{q}.\tag{S\^{P}1}
\end{align}
When subproblem \eqref{eq:sp1} is projected onto the Fourier domain as in \eqref{eq:sp1hat}, it is simplified in two ways. On one hand, the block diagonalized matrices have been reduced to a maximum of $Tn(\max\lbrace n_u,pn_y\rbrace)^2$ nonzero elements, where $T$ is the prediction horizon of the MPC problem. On the other, the pattern of complex conjugates allows vectors and matrices to be truncated as in Theorem \ref{thm:decompblockqp}.

 The modified subproblem \eqref{eq:sp2_sat} reads as
\begin{align}\label{eq:sp2hat_sat}
\hat{v}^i = \trunc_{N_v}\lbrace \psi_v^H\sat_{[\ubar{v},\bar{v}]}\left\lbrace \psi_v\aug_{N_v}\lbrace\hat{K}\hat{z}^i+\rho^{-1}\trunc_{N_v}\hat{\gamma}^{i-1}\rbrace\right\rbrace\rbrace.\nonumber\tag{S\^{P}2}
\end{align}
The discrete Fourier transformations required in subproblem \eqref{eq:sp2hat_sat} are the main drawbacks of algorithm \eqref{alg:circ_admm} and the problem sizes required to outperform algorithm \eqref{alg:admm} are discussed in section \ref{sec:complexity}.

 The decomposed and truncated dual variables are updated using
\begin{align}\label{eq:sp3hat}
\trunc_{N_v}\hat{\gamma}^i = \trunc_{N_v}\hat{\gamma}^{i-1} +\rho(\hat{K}\hat{z}^i-\hat{y}^i). \nonumber\tag{S\^{P}3}
\end{align}
As mentioned in the proof of Theorem \ref{thm:truncblockqp}, $\hat{\gamma}\in\Sc_v$ and it is therefore sufficient to update the truncated dual variable. In practice, the term $\hat{K}\hat{z}^i$ can be cached during \eqref{eq:sp2hat_sat} and reused in \eqref{eq:sp3hat}.

 Finally, the convergence of Algorithm \ref{alg:circ_admm} is verified by
\begin{align}\label{eq:convergence}
\twonorm{\aug_{N_v}\lbrace\hat{v}^i-\hat{v}^{i-1}\rbrace}^2 < \epsilon,\\
\twonorm{\hat{\gamma}^i-\hat{\gamma}^{i-1}}^2 < \epsilon.
\end{align}
To see that this criterion is equivalent to the one of Algorithm \ref{alg:admm}, note that $\twonorm{\aug_{N_v}\lbrace \hat{v}\rbrace}=\twonorm{\psi_v^Hv}=\twonorm{v}$. In practice, the augmentation is not required as we can relate $\twonorm{v}^2$ and $\twonorm{\hat{v}}^2$ by
\begin{align}
2\twonorm{\hat{v}}^2-\sum_{k=1}^{N_v}(\hat{v}^k_1)^\Tr\hat{v}^k_1 = \twonorm{v}^2,
\end{align}
where we assumed that $n$ is odd and where the sum over the segments of $\hat{v}$ extracts the first (real) block of each segment. Note that for the returned value $u(t)$ in the last step of Algorithm \ref{alg:circ_admm} it is assumed that $n$ is odd as well.

 As shown in Theorems \ref{thm:decompblockqp} and \ref{thm:truncblockqp}, Algorithms \ref{alg:admm} and \ref{alg:circ_admm} produce an equivalent solution in the sense that the minimizers $z^*$ and $\hat{z}^*$ are related by $z^*=\psi_z\aug_{N_z}\lbrace\hat{z}^*\rbrace$.

\subsection{Computational Complexity}\label{sec:complexity}
In this section, the computational complexities of Algorithm \ref{alg:admm} and \ref{alg:circ_admm} are compared. It is assumed that the matrices in Definition \ref{def:blockcircqp} are full and that $\inR{J}{(Nn_zn)}{(Nn_zn)}$ and $\inR{K}{(Nn_vn)}{(Nn_zn)}$ are partitioned into segments of identical lengths $n_z$ and $n_v$, respectively, with $n_v=n_z$. Algorithm \ref{alg:circ_admm} makes use of discrete Fourier transformations during initialization, finalization and in subproblem \eqref{eq:sp2hat_sat}. These can be performed using the identity \eqref{eq:PsiPerm} and a sequence of fast Fourier transformations (FFT). For a vector of length $r$, the complexity of the FFT is $\co{r\log r}$.

 For both algorithms, the main burden lies in solving the linear system in subproblems \eqref{eq:sp1} and \eqref{eq:sp1hat}. Depending on the structure of the linear system, it can be solved in numerous ways. For simplicity, it is assumed that the linear system is solved using a matrix inverse that is pre-computed offline. In that case, \eqref{eq:sp1} is of complexity $\co{(Nnn_z)^2}$, while the block diagonalization and truncation in \eqref{eq:sp1hat} reduces the complexity to $\co{4\frac{n}{2}(Nn_z)^2}$, where the factor $4$ accounts for the required complex arithmetic. Because the operations of the saturation function are negligible, \eqref{eq:sp2_sat} accounts for a complexity of $\co{(Nnn_z)^2+Nnn_z}$. Note that the term $Kz^i$ is reused in \eqref{eq:sp3}. As the projection onto the boundaries $[\ubar{v},\bar{v}]$ must be carried out in the original domain, the drawbacks of the Fourier transformation become evident in \eqref{eq:sp2hat_sat}. This results in a complexity of $\co{4\frac{n}{2}(Nn_z)^2+2Nn_zn\log n}$. Lastly, \eqref{eq:sp3} and \eqref{eq:sp3hat} are of complexities $\co{Nnn_z}$ and $\co{2N\frac{n}{2}n_z}$, respectively.

 The total complexities of Algorithm \ref{alg:admm} and \ref{alg:circ_admm} are summarized in Table \ref{tbl:complexity}. The totals reveal that under the assumptions of this section, Algorithm \ref{alg:circ_admm} is cheaper than \ref{alg:admm} if $n>2$ and $Nn_z>3$.

\begin{table}[!t]
\renewcommand{\arraystretch}{1.3}
\caption{Computational Complexity for Algorithms \ref{alg:admm} and \ref{alg:circ_admm}}
\label{tbl:complexity}
\centering
  \begin{tabular}{@{}r  l  l@{}}
    \hline
                   & Algorithm \ref{alg:admm} & Algorithm \ref{alg:circ_admm}\\
    \hline
    \eqref{eq:sp1} & $\co{(Nnn_z)^2}$     & $\co{2n(Nn_z)^2}$ \\
    \eqref{eq:sp2_sat} & $\co{(Nnn_z)^2}$ & $\co{2n(Nn_z)^2+2Nn_zn\log n}$ \\
    \eqref{eq:sp3} & $\co{Nnn_z}$         & $\co{Nnn_z}$ \\
    \hline
    Total          & $\co{Nnn_z(2Nnn_z+1)}$ & $\co{Nnn_z(4Nn_z+1+2\log n)}$\\
    \hline
  \end{tabular}
\end{table}

\section{Simulations}\label{sec:simulations}

\subsection{Random Constrained QP with Block Circulant Blocks}\label{sec:random_qp}
We first consider a set of randomly generated constrained block circulant QPs to gauge performance of Algorithm \ref{alg:circ_admm}. Figure \ref{fig:decomposition} shows the decomposed left-hand side of \eqref{eq:sp1}, $J+\rho K^\Tr K$, of a QP of order $n=4$ before applying the truncation. On one hand, it is evident how the transformation matrices $\psi_z$ and $\psi_v$ block diagonalize $J+\rho K^\Tr K$. On the other, the pattern of complex conjugate blocks motivating the truncation becomes apparent.
\begin{figure}[!t]
\centering
\includegraphics[scale=1]{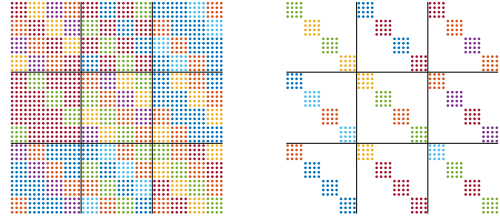}
\caption{Sparsity patterns of $J+\rho K^\Tr K$ (left) and $\tilde{J}+\rho \tilde{K}^H \tilde{K}$ (right) for $N_z=N_v=3$ and $l_z^j=l_v^j=4,\,j=1,2,3$. The colors are proportional to the magnitude of the matrix elements.}
\label{fig:decomposition}
\end{figure}

\ifx\uselargeplots\undefined
	\begin{figure}[!t]\label{fig:resultsqp}
	\centering
	\begin{overpic}[scale=1]{plots/results_qp_3_logy}
	\put(45,-3){\footnotesize{Order $n$}}
	\put(-4,48){\rotatebox{90}{\footnotesize{\eqref{eq:sp1}-\eqref{eq:sp3}}}}
	\put(-4,32){\rotatebox{90}{\footnotesize{\eqref{eq:sp1}}}}
	\put(-4,9.5){\rotatebox{90}{\footnotesize{\eqref{eq:sp2_sat}}}}
	\end{overpic}
	\caption{Logarithmically scaled execution times of Algorithms \ref{alg:admm} (continuous) and \ref{alg:circ_admm} (dashed) in milliseconds for random block circulant QPs. The first row shows the total time for \eqref{eq:sp1}-\eqref{eq:sp3}. The second and third row show the execution times for \eqref{eq:sp1} and \eqref{eq:sp2_sat}, respectively.}
	\end{figure}
\else
	\begin{figure}[!t]\label{fig:resultsqp}
	\centering
	\begin{overpic}[scale=1]{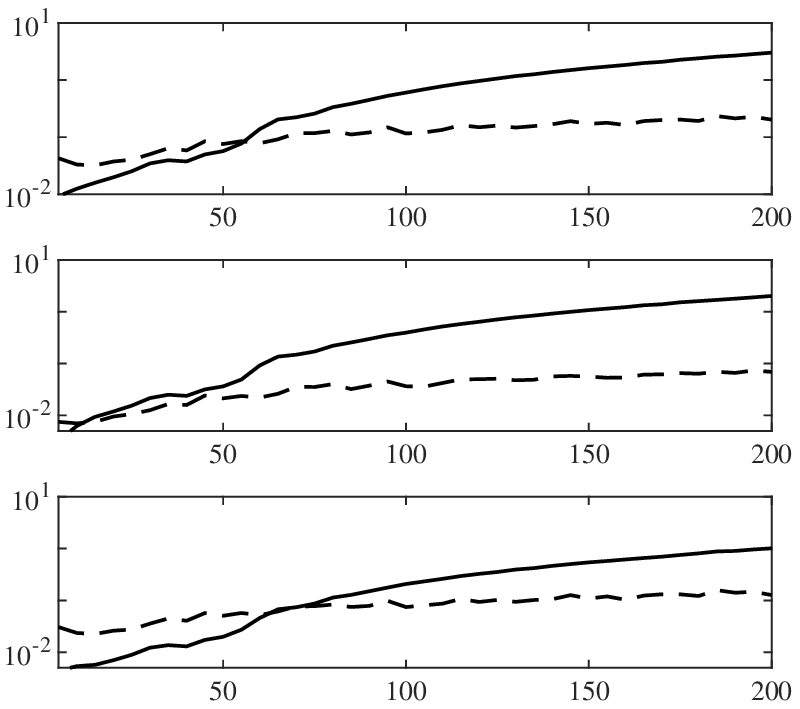}
	\put(45,-3){\footnotesize{Order $n$}}
	\put(-2.5,66.5){\rotatebox{90}{\footnotesize{\eqref{eq:sp1}-\eqref{eq:sp3}}}}
	\put(-2.5,42.5){\rotatebox{90}{\footnotesize{\eqref{eq:sp1}}}}
	\put(-2.5,13){\rotatebox{90}{\footnotesize{\eqref{eq:sp2_sat}}}}
	\end{overpic}
	\caption{Logarithmically scaled execution times of Algorithms \ref{alg:admm} (continuous) and \ref{alg:circ_admm} (dashed) in milliseconds for random block circulant QPs. The first row shows the total time for \eqref{eq:sp1}-\eqref{eq:sp3}. The second and third row show the execution times for \eqref{eq:sp1} and \eqref{eq:sp2_sat}, respectively.}
	\end{figure}
\fi

 Figure \ref{fig:resultsqp} shows the execution times of Algorithms \ref{alg:admm} and \ref{alg:circ_admm} for random block circulant QPs of increasing order $n$ with $N=1$ and $l_z=l_v=10$. Even though the analysis of Section \ref{sec:complexity} accounted for the required Fourier transformations, Figure \ref{fig:resultsqp} reveals that for small $n$ the additional operations required in \eqref{eq:sp2hat_sat} such as truncating, permuting and augmenting the vectors are not negligible. For larger $n$ these side effects lose their significance and the superiority of Algorithm \ref{alg:circ_admm} becomes evident.

\subsection{Ring of Masses}\label{sec:oscillators}
Consider $n$ identical masses arranged in a ring of radius $R$ and connected through springs and dampers. At equilibrium the masses are uniformly spaced around the ring at angles $\phi_j=\frac{2\pi}{n}$ for $j=1,\dots,n$. The dynamics of a small deviation $\Delta\phi_j$ of mass $j$ from the equilibrium angle $\phi_j$ are obtained as
\begin{align}\label{eq:oscillator}
\begin{split}
\Delta\ddot{\phi}_j = &\frac{k}{m}(\Delta\phi_{j+1}+\Delta\phi_{j-1}-2\Delta\phi_{j})\\&+ \frac{d}{m}(\Delta\dot{\phi}_{j+1}+\Delta\dot{\phi}_{j-1}-2\Delta\dot{\phi}_{j})+\frac{1}{m}T_j,
\end{split}
\end{align}
where all indices are modulo $n$, $T_j$ is a controllable torque acting on each mass, $m$ is the mass and $k$ and $d$ are the spring and damper coefficients, respectively.
Discretizing the dynamics w.r.t. time and setting $x_k=(\Delta\phi_1^k,\Delta\dot{\phi}^k_1,\dots,\Delta\phi_n^k,\Delta\dot{\phi}^k_n)^\Tr$ yields a discrete-time block circulant dynamical system. By assigning $u_k=(T_1^k,\dots,T_n^k)^\Tr$, a block circulant MPC problem of order $n$  with $n_u=1$ and $n_x=2$ can be defined. For the following experiments, matrices $Q$ and $R$ were chosen to be identity matrices, the prediction horizon $T$ was set to $10$ and the states $x_k$ and the inputs $u_k$ were constrained separately, i.e. $C_1 = \I_{nn_x}$, $D_1=0$ and $C_2=0$, $D_2=\I_{nn_u}$. According to Corollary \ref{thm:blocpmpisblockqp}, the latter MPC problem defines a CBCQP.

\ifx\uselargeplots\undefined
	\begin{figure}[!t]\label{fig:results_oscillator}
	\centering
	\begin{overpic}[scale=1]{plots/results_oscillator_4_logy}
	\put(45,-3){\footnotesize{Order $n$}}
	\put(-4,48){\rotatebox{90}{\footnotesize{\eqref{eq:sp1}-\eqref{eq:sp3}}}}
	\put(-4,32){\rotatebox{90}{\footnotesize{\eqref{eq:sp1}}}}
	\put(-4,10){\rotatebox{90}{\footnotesize{\eqref{eq:sp2_sat}}}}
	\end{overpic}
	\caption{Logarithmically scaled execution times of Algorithms \ref{alg:admm} (continuous) and \ref{alg:circ_admm} (dashed) in milliseconds for the MPC problem of Section \ref{sec:oscillators}. The first row shows the total time for \eqref{eq:sp1}-\eqref{eq:sp3}. The second and third row show the execution times for \eqref{eq:sp1} and \eqref{eq:sp2_sat}, respectively.}
	\end{figure}
\else
	\begin{figure}[!t]\label{fig:results_oscillator}
	\centering
	\begin{overpic}[scale=1]{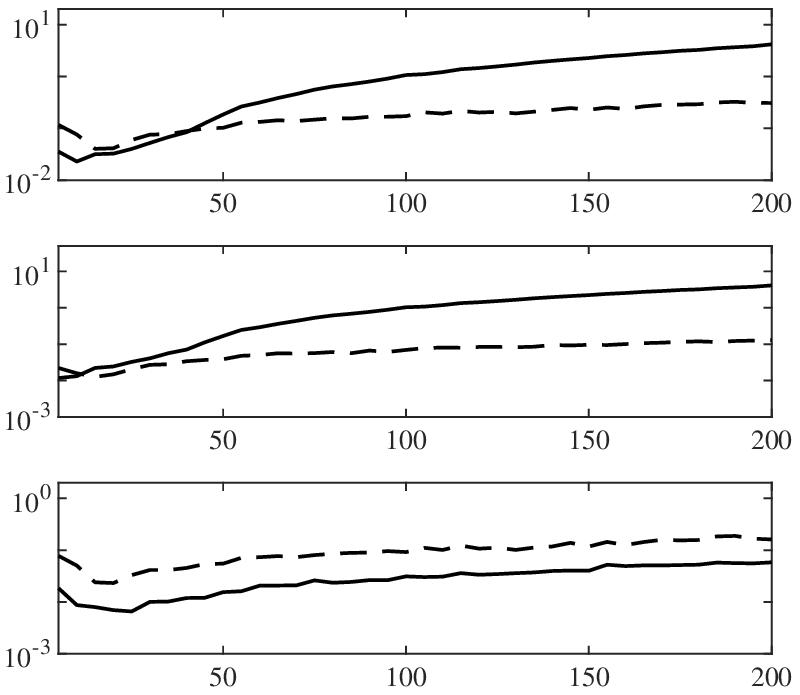}
	\put(45,-3){\footnotesize{Order $n$}}
	\put(-2.5,66.5){\rotatebox{90}{\footnotesize{\eqref{eq:sp1}-\eqref{eq:sp3}}}}
	\put(-2.5,42.5){\rotatebox{90}{\footnotesize{\eqref{eq:sp1}}}}
	\put(-2.5,13){\rotatebox{90}{\footnotesize{\eqref{eq:sp2_sat}}}}
	\end{overpic}
	\caption{Logarithmically scaled execution times of Algorithms \ref{alg:admm} (continuous) and \ref{alg:circ_admm} (dashed) in milliseconds for the MPC problem of Section \ref{sec:oscillators}. The first row shows the total time for \eqref{eq:sp1}-\eqref{eq:sp3}. The second and third row show the execution times for \eqref{eq:sp1} and \eqref{eq:sp2_sat}, respectively.}
	\end{figure}
\fi

 Figure \ref{fig:results_oscillator} compares the performance of Algorithm \ref{alg:admm} and \ref{alg:circ_admm} solving the MPC problem with random initial conditions for an increasing number of masses $n$. Each problem was solved with an average number of 25 ADMM iterations. As for the example of Section \ref{sec:random_qp}, Algorithm \ref{alg:circ_admm} performs worse than Algorithm \ref{alg:admm} for small $n$ when the drawbacks of the Fourier transformation in \eqref{eq:sp2hat_sat} and other side effects outweigh the computational gains in \eqref{eq:sp1hat}. The benefits of the Fourier transformation become evident for larger $n$.

\section{Conclusions}
\balance
This paper demonstrated how to exploit the particular structure of a MPC problem for block circulant systems. Based on the properties of block circulant matrices, a block circulant MPC problem was defined and connected to a general constrained QP with block circulant blocks. A transformation was derived which block diagonalizes any constrained QP with block circulant blocks and allows to truncate the transformed vectors. A modified ADMM algorithm for the transformed and truncated system was developed. The modified ADMM algorithm was tested using a series of random constrained QPs with block-circulant problem data and using an academic example of a block-circulant MPC problem. In both cases, the evaluation of the results revealed that the modified ADMM algorithm performs significantly better for increasing problem sizes.

\bibliographystyle{IEEEtran}
\bibliography{IEEEabbrv,bib_circ_admm}

\end{document}